\documentclass{article}
\usepackage[utf8]{inputenc}
\usepackage{mathtools}
\usepackage{amsthm}
\usepackage{cite}
\usepackage{authblk}

\newtheorem{lemma}{Lemma}
\newtheorem{theorem}{Theorem}
\DeclarePairedDelimiter\ceil{\lceil}{\rceil}
\DeclarePairedDelimiter\floor{\lfloor}{\rfloor}
\begin{document}
\title{Bounds for the Graham-Pollak Theorem for Hypergraphs}
\author{Anand Babu  \ \ \ \  Sundar Vishwanathan}
\affil{Department of Computer Science {\&} Engineering, \\ IIT Bombay}
\date{}
\maketitle
\begin{abstract}
Let $f_r(n)$ represent the minimum number of complete $r$-partite $r$-graphs required to partition the edge set of the complete $r$-uniform hypergraph on $n$ vertices. The Graham-Pollak theorem states that $f_2(n)=n-1$.\\
An upper bound of $(1+o(1)){n \choose \lfloor{\frac{r}{2}}\rfloor}$ was known. Recently this was improved to $\frac{14}{15}(1+o(1)){n \choose \lfloor{\frac{r}{2}}\rfloor}$ for even $r \geq 4$. A bound of $\bigg[\frac{r}{2}(\frac{14}{15})^{\frac{r}{4}}+o(1)\bigg](1+o(1)){n \choose \lfloor{\frac{r}{2}}\rfloor}$ was also proved recently. The smallest odd $r$ for which $c_r < 1$ that was known was for $r=295$.  In this note we improve this to $c_{113}<1$ and also give better upper bounds for $f_r(n)$, for small values of even $r$.
\end{abstract}
\section{Introduction}
An $r$-uniform hypergraph $H$ (also referred to as an $r$-graph) is said to be $r$-partite if its vertex set $V(H)$ can be partitioned into sets $V_1, V_2, \cdots ,V_r$, so that every edge in the edge set $E(H)$ of $H$ consists of choosing precisely one vertex from each set $V_i$. That is, $E(H) \subseteq V_1 \times V_2 \times \cdots \times V_r$. Let $f_r(n)$ be the minimum number of complete $r$-partite $r$-graphs needed to partition the edge set of the complete $r$-uniform hypergraph on $n$ vertices. The problem of determining $f_r(n)$ for $r>2$ was proposed by Aharoni and Linial\cite{alon1986decomposition}. For $r=2$, $f_2(n)$ is the minimum number of bipartite subgraphs required to partition the edge set of the complete graph. Graham and Pollak(\cite{graham1971addressing,graham1972embedding} see also \cite{graham1978distance}) proved that at least $n-1$ bipartite graphs are required to cover a complete graph. Other proofs were found by Tverbeg\cite{tverberg1982decomposition}, Peck\cite{peck1984new} and Vishwanathan\cite{vishwanathan2008polynomial,vishwanathan2013counting}.
\paragraph*{}
For a general $r$, constructions due to Alon \cite{alon1986decomposition} and later Cioab{\u{a}} et.al  \cite{cioabua2009decompositions} give an upper bound for $f_r(n)$. Cioab{\u{a}} et.al showed that by ordering the vertices and then by considering the collection of $r$-graphs whose even positions are fixed, partitions the edge set of the complete $r$-uniform hypergraph. The cardinality of the collection of $r$-graphs obtained so is ${n-(r+1)/2 \choose (r-1)/2}$ for odd $r$, and ${n-r/2 \choose r/2}$ for even $r$. The upper bound described below is from the above construction and the lower bound is obtained using the ideas from linear algebra. 
\begin{equation*}
\begin{split}
\frac{2}{{2\floor {r/2} \choose \floor{r/2}}}(1+o(1)){n \choose \floor{\frac{r}{2}}}\leq f_r(n)\leq (1-o(1)){n \choose \floor{\frac{r}{2}}}.\\
\end{split}
\end{equation*}
Alon also proved that $f_3(n)=n-2$ \cite{alon1986decomposition}. Cioab{\u{a}} and Tait \cite{cioabua2013variations} showed that the construction is not tight in general but there was no asymptotic improvement to Alon's bound. In a breakthrough paper,  Leader, Mili{\'c}evi{\'c} and Tan \cite{leader2017decomposing} showed that $f_4(n) \leq (\frac{14}{15})(1+o(1)){n \choose 2}$. Using this they observed that $f_r(n) \leq (\frac{14}{15})(1+o(1)){n \choose \frac{r}{2}}$ for even $r$. Later, Leader and Tan\cite{leader2017improved} showed that for a general $r\geq 4$, $f_r(n)\leq c_r(1+o(1)){n \choose \floor{\frac{r}{2}}}$  where $c_r \leq \frac{r}{2}(\frac{14}{15})^{\frac{r}{4}}+o(1)$ and as a direct consequence showed that $c_{295}< 1$ \cite{leader2017improved}. The smallest odd $r_0$ for which $c_{r_0}<1$ is important since this implies that $c_r < 1$ for all $r>r_0$. In this note we improve the smallest known odd $r$, for which $c_r <1$ to $r=113$. We also give an improved upper bound for $f_r(n)$ for even $r$ and $8\leq r \leq 1096$ which is used in the above result. We show that for all even $r\geq 6$,
\begin{equation*}
\begin{split}
f_r(n)\leq (\frac{14}{15})^{\frac{r}{6}}(1+o(1)){n \choose \floor{\frac{r}{2}}}\\
\end{split}
\end{equation*}
\section{The Main Result}
Let $S$ and $T$ be two disjoint sets. Let ${{S} \choose {a}}\times {{T} \choose {b}}$ denote all subsets $X$ of $S \cup T$ s.t. $|X \cap S|=a$ and $|X \cap T|=b$.\\

A set $\Gamma$ of complete $r$-partite $r$-graphs over $S \cup T$ is said to {\em exactly cover} a hypergraph $F$, if the hypergraphs in $\Gamma$ are edge-disjoint and the union of the edges of the hypergraphs in $\Gamma$ is $F$. A complete $r$-partite $r$-graph is also referred to as a {\em block}.\\

Let $f_r(n)$ denote the minimum number of complete $r$-partite $r$-graphs required to exactly cover the edge set of the complete $r$-uniform hypergraph.

\begin{theorem}
For even $r\geq 6$, $f_r(n) \leq (\frac{14}{15})^\frac{r}{6}(1+o(1)) {n \choose \floor{\frac{r}{2}}}$.\\
(Here the o(1) term is as $n \rightarrow \infty$ with r fixed.)
\end{theorem}
\begin{proof}
We show that for even $m\geq 8$, and $n \geq m$,
\begin{equation*}
\begin{split}
f_m(n) &\leq (\frac{14}{15})^\frac{m}{6} \frac{n^{\frac{m}{2}}}{(\frac{m}{2})!}+n^{\frac{m}{2}-1} \log n
\end{split}
\end{equation*}
The proof is by induction on $m$ and $n$.\\
We use the following known bounds. $f_2(n) \leq n-1$, $f_3(n) \leq n-2$.\\
Suppose $m$ is a multiple of $4$. By dividing the set $[n]$ into two parts of size $\frac{n}{2}$ each, we get the following recurrence for $f_m(n)$.\\
\begin{equation*}
\begin{split}
f_m(n) &\leq 2\cdot f_m(\frac{n}{2})+ 2 \cdot f_{m-1}(\frac{n}{2})+ 2 \cdot f_{m-2}(\frac{n}{2})f_2(\frac{n}{2})+ 2\cdot f_{m-3}(\frac{n}{2})\cdot f_{3}(\frac{n}{2})+...+[f_{\frac{m}{2}}(\frac{n}{2})]^2\\
\end{split}
\end{equation*}

The bound $f_4(n) \leq (\frac{14}{15})\frac{n^2}{2!}+n\log n$, follows from \cite{leader2017decomposing}. We prove that $f_6(n) \leq (\frac{14}{15})\frac{n^3}{3!}+n^2 \log n$ first. The base case for $f_6(n)$ holds since $f_6(6)=1$. Assume it is true for all values less than $n$. 
\begin{equation*}
\begin{split}
f_6(n) &\leq 2f_6(\frac{n}{2})+2f_5(\frac{n}{2})+2f_2(\frac{n}{2})f_4(\frac{n}{2})+f_3(\frac{n}{2})f_3(\frac{n}{2})\\
&\leq 2\bigg[(\frac{14}{15})\frac{n^3}{8\cdot 3!}+\frac{n^2\log (\frac{n}{2})}{4}\bigg]+2\frac{n^2}{4\cdot 2!}+2\frac{n}{2}\bigg[(\frac{14}{15})\frac{n^2}{4\cdot 2!}+\frac{n\log (\frac{n}{2})}{2}\bigg]+\frac{n^2}{4}\\
&\leq (\frac{14}{15})\frac{n^3}{3!}+n^2 \log n\\
\end{split}
\end{equation*}
Now we prove that $f_m(n) \leq (\frac{14}{15})^\frac{m}{6} \frac{n^{\frac{m}{2}}}{(\frac{m}{2})!}+n^{\frac{m}{2}-1} \log n$. The base case for $f_m(n)$ holds since $f_m(m)=1$. Assume it is true for all values less than $n$. Since
\begin{equation*}
\begin{split}
f_m(n) &\leq 2\cdot f_m(\frac{n}{2})+ 2 \cdot f_{m-1}(\frac{n}{2})+ 2 \cdot f_{m-2}(\frac{n}{2})f_2(\frac{n}{2})+ 2\cdot f_{m-3}(\frac{n}{2})\cdot f_{3}(\frac{n}{2})+...+[f_{\frac{m}{2}}(\frac{n}{2})]^2\\
\end{split}
\end{equation*}
as stated earlier. By rearranging the terms according to even and odd indices we have,
\begin{equation*}
\begin{split}
f_m(n)&\leq 2\cdot f_m(\frac{n}{2})+ 2 \cdot f_{m-2}(\frac{n}{2})f_2(\frac{n}{2})+\cdots+[f_{\frac{m}{2}}(\frac{n}{2})]^2\\
&\qquad +2 \cdot f_{m-1}(\frac{n}{2})+2\cdot f_{m-3}(\frac{n}{2})\cdot f_{3}(\frac{n}{2})+\cdots+2\cdot f_{\frac{m}{2}+1}(\frac{n}{2})\cdot f_{\frac{m}{2}-1}(\frac{n}{2})\\
\end{split}
\end{equation*}
Substituting for $f_2(\frac{n}{2})$, $f_4(\frac{n}{2})$ and $f_6(\frac{n}{2})$ and using the inductive hypothesis for all even $i\geq 8$ and $f_i(\frac{n}{2})={\frac{n}{2} \choose \floor{\frac{i}{2}}}$ for all odd $i \geq 3$, we get
\begin{equation*}
\begin{split}
f_m(n)&\leq 2\bigg[(\frac{14}{15})^\frac{m}{6} \frac{n^{\frac{m}{2}}}{2^{\frac{m}{2}}(\frac{m}{2})!}+\frac{n^{\frac{m}{2}-1}\log (\frac{n}{2})}{2^{\frac{m}{2}-1}}\bigg] +2 (\frac{n}{2})\bigg[(\frac{14}{15})^\frac{m-2}{6} \frac{n^{\frac{m}{2}-1}}{2^{\frac{m}{2}-1}(\frac{m}{2}-1)!}+\frac{n^{\frac{m}{2}-2}\log (\frac{n}{2})}{2^{\frac{m}{2}-2}}\bigg]\\
&\qquad+2 \bigg[(\frac{14}{15})\frac{n^2}{2^2 \cdot 2!}+\frac{n\log (\frac{n}{2})}{2}\bigg]\bigg[(\frac{14}{15})^\frac{m-4}{6} \frac{n^{\frac{m}{2}-2}}{2^{\frac{m}{2}-2}(\frac{m}{2}-2)!}+\frac{n^{\frac{m}{2}-3}\log (\frac{n}{2})}{2^{\frac{m}{2}-3}}\bigg]\\
&\qquad +2 \bigg[(\frac{14}{15}) \frac{n^3}{2^3\cdot 3!}+\frac{n^2\log (\frac{n}{2})}{2^2}\bigg]\bigg[(\frac{14}{15})^{\frac{m}{6}-1} \frac{n^{\frac{m}{2}-3}}{2^{\frac{m}{2}-3}(\frac{m}{2}-3)!}+\frac{n^{\frac{m}{2}-4}\log (\frac{n}{2})}{2^{\frac{m}{2}-4}}\bigg]+\cdots\\
&\qquad +\bigg[(\frac{14}{15})^\frac{m}{12} \frac{n^{\frac{m}{4}}}{2^{\frac{m}{4}}(\frac{m}{4})!}+\frac{n^{\frac{m}{4}-1}\log (\frac{n}{2})}{2^{\frac{m}{4}-1}}\bigg]\bigg[(\frac{14}{15})^\frac{m}{12} \frac{n^{\frac{m}{4}}}{2^{\frac{m}{4}}(\frac{m}{4})!}+\frac{n^{\frac{m}{4}-1}\log (\frac{n}{2})}{2^{\frac{m}{4}-1}}\bigg]\\
&\qquad +\bigg[2 {\frac{n}{2} \choose \frac{m}{2}-1} + 2 {{\frac{n}{2}} \choose {\frac{m}{2}}-2}{{\frac{n}{2}} \choose 1}+\cdots + 2{{\frac{n}{2}} \choose {\frac{m}{4}}-1}{{\frac{n}{2}} \choose {\frac{m}{4}}}\bigg]\\
\end{split}
\end{equation*}
Grouping terms according to their asymptotic behavior of $n$ and since there are at most $\frac{m}{4}$ terms, each contributing $\frac{2n^{\frac{m}{2}-2}\log ^2 (\frac{n}{2})}{2^{\frac{m}{2}-2}}$, we have
\begin{equation*}
\begin{split}
f_m(n)&\leq \bigg[2(\frac{14}{15})^\frac{m}{6} \frac{n^{\frac{m}{2}}}{2^{\frac{m}{2}}(\frac{m}{2})!}+2(\frac{14}{15})^\frac{m-2}{6} \frac{n^{\frac{m}{2}}}{2^{\frac{m}{2}}(\frac{m}{2}-1)!}+2(\frac{14}{15})^\frac{m+2}{6} \frac{n^{\frac{m}{2}}}{2^{\frac{m}{2}}\cdot 2!\cdot (\frac{m}{2}-2)!}+ \cdots \\
&\qquad + 2(\frac{14}{15})^\frac{m}{6} \frac{n^{\frac{m}{2}}}{2^{\frac{m}{2}}\cdot (\frac{i}{2})!(\frac{m}{2}-\frac{i}{2})!}+\cdots + (\frac{14}{15})^\frac{m}{6} \frac{n^{\frac{m}{2}}}{2^{\frac{m}{2}}(\frac{m}{4})!(\frac{m}{4})!}\bigg]\\
&\qquad +\frac{2n^{\frac{m}{2}-1}\log(\frac{n}{2})}{2^{\frac{m}{2}-1}}+\frac{2n^{\frac{m}{2}-1}\log(\frac{n}{2})}{2^{\frac{m}{2}-1}}+2\bigg[ (\frac{14}{15}) \frac{n^{\frac{m}{2}-1}\log(\frac{n}{2})}{2^{\frac{m}{2}-1}\cdot 2!}+ (\frac{14}{15})^{\frac{m-4}{6}}\frac{n^{\frac{m}{2}-1}\log(\frac{n}{2})}{2^{\frac{m}{2}-1}\cdot(\frac{m}{2}-2)!}\bigg]\\
&\qquad + 2\bigg[ (\frac{14}{15}) \frac{n^{\frac{m}{2}-1}\log(\frac{n}{2})}{2^{\frac{m}{2}-1}\cdot 3!}+ (\frac{14}{15})^{\frac{m-6}{6}}\frac{n^{\frac{m}{2}-1}\log(\frac{n}{2})}{2^{\frac{m}{2}-1}\cdot(\frac{m}{2}-3)!}\bigg]+\cdots\\
&\qquad +\bigg[ (\frac{14}{15})^{\frac{m}{12}} \frac{n^{\frac{m}{2}-1}\log(\frac{n}{2})}{2^{\frac{m}{2}-1}\cdot (\frac{m}{4})!}+ (\frac{14}{15})^{\frac{m}{12}} \frac{n^{\frac{m}{2}-1}\log(\frac{n}{2})}{2^{\frac{m}{2}-1}\cdot (\frac{m}{4})!}\bigg]\\
&\qquad +\frac{2 n^{\frac{m}{2}-2}\log ^2 (\frac{n}{2})}{2^{\frac{m}{2}-2}}\frac{m}{4}\\
&\qquad +\bigg[2\frac{n^{\frac{m}{2}-1}}{2^{\frac{m}{2}-1}\cdot (\frac{m}{2}-1)!}+ 2\frac{n^{\frac{m}{2}-1}}{2^{\frac{m}{2}-1}\cdot (\frac{m}{2}-2)!}+\cdots+ 2\frac{n^{\frac{m}{2}-1}}{2^{\frac{m}{2}-1}\cdot (\frac{m}{4})! (\frac{m}{4}-1)!}\bigg]\\
\end{split}
\end{equation*}
Since $(\frac{14}{15})^k <1$, for positive $k \geq 0$, for the terms of asymptotic order $n^{\frac{m}{2}-1}\log n$, we have
\begin{equation*}
\begin{split}
f_m(n)&\leq \bigg[2(\frac{14}{15})^\frac{m}{6} \frac{n^{\frac{m}{2}}}{2^{\frac{m}{2}}(\frac{m}{2})!}+2(\frac{14}{15})^\frac{m-2}{6} \frac{n^{\frac{m}{2}}}{2^{\frac{m}{2}}(\frac{m}{2}-1)!}+2(\frac{14}{15})^\frac{m+2}{6} \frac{n^{\frac{m}{2}}}{2^{\frac{m}{2}}\cdot 2!\cdot (\frac{m}{2}-2)!}+\cdots \\
&\qquad + 2(\frac{14}{15})^\frac{m}{6} \frac{n^{\frac{m}{2}}}{2^{\frac{m}{2}}\cdot (\frac{i}{2})!(\frac{m}{2}-\frac{i}{2})!}+\cdots + (\frac{14}{15})^\frac{m}{6} \frac{n^{\frac{m}{2}}}{2^{\frac{m}{2}}(\frac{m}{4})!(\frac{m}{4})!}\bigg]\\
&\qquad +\bigg[\frac{2n^{\frac{m}{2}-1}\log(\frac{n}{2})}{2^{\frac{m}{2}-1}}+\frac{2n^{\frac{m}{2}-1}\log(\frac{n}{2})}{2^{\frac{m}{2}-1}}+\frac{2n^{\frac{m}{2}-1}\log(\frac{n}{2})}{2^{\frac{m}{2}-1}\cdot 2!}+ \frac{2n^{\frac{m}{2}-1}\log(\frac{n}{2})}{2^{\frac{m}{2}-1}\cdot(\frac{m}{2}-2)!}\\
&\qquad + \frac{2n^{\frac{m}{2}-1}\log(\frac{n}{2})}{2^{\frac{m}{2}-1}\cdot 3!}+ \frac{2n^{\frac{m}{2}-1}\log(\frac{n}{2})}{2^{\frac{m}{2}-1}\cdot(\frac{m}{2}-3)!}+\cdots\\
&\qquad +\frac{n^{\frac{m}{2}-1}\log(\frac{n}{2})}{2^{\frac{m}{2}-1}\cdot (\frac{m}{4})!}+ \frac{n^{\frac{m}{2}-1}\log(\frac{n}{2})}{2^{\frac{m}{2}-1}\cdot (\frac{m}{4})!}\bigg]\\
&\qquad +\frac{2 n^{\frac{m}{2}-2}\log ^2 (\frac{n}{2})}{2^{\frac{m}{2}-2}}\frac{m}{4}\\
&\qquad +\bigg[2\frac{n^{\frac{m}{2}-1}}{2^{\frac{m}{2}-1}\cdot (\frac{m}{2}-1)!}+ 2\frac{n^{\frac{m}{2}-1}}{2^{\frac{m}{2}-1}\cdot (\frac{m}{2}-2)!}+\cdots+ 2\frac{n^{\frac{m}{2}-1}}{2^{\frac{m}{2}-1}\cdot (\frac{m}{4})! (\frac{m}{4}-1)!}\bigg]\\
\end{split}
\end{equation*}

Using the identity $\sum_{i=0}^{N} \frac{1}{i!\cdot (N-i)!}=\frac{2^N}{N!}$, on the terms in the first and last square braces we get,
\begin{equation*}
\begin{split}
f_m(n)&\leq \bigg[(\frac{14}{15})^{\frac{m}{6}}\frac{n^{\frac{m}{2}}}{(\frac{m}{2})!}-2(\frac{14}{15})^{\frac{m}{6}}\frac{n^{\frac{m}{2}}}{2^{\frac{m}{2}}(\frac{m}{2}-1)!}-2(\frac{14}{15})^{\frac{m}{6}}\frac{n^{\frac{m}{2}}}{2^{\frac{m}{2}}\cdot 2! \cdot(\frac{m}{2}-2)!}+2(\frac{14}{15})^{\frac{m-2}{6}}\frac{n^{\frac{m}{2}}}{2^{\frac{m}{2}}(\frac{m}{2}-1)!}\\
&\qquad +2(\frac{14}{15})^{\frac{m+2}{6}}\frac{n^{\frac{m}{2}}}{2^{\frac{m}{2}}\cdot 2!\cdot(\frac{m}{2}-2)!}\bigg]+\bigg[\frac{2n^{\frac{m}{2}-1}\log (\frac{n}{2})}{2^{\frac{m}{2}-1}}\bigg]\bigg[1+1+\frac{1}{2!}+\cdots +\frac{1}{(\frac{m}{2}-2)!}\bigg]\\
&\qquad + \frac{2 n^{\frac{m}{2}-2}\log ^2 (\frac{n}{2})}{2^{\frac{m}{2}-2}}\frac{m}{4} + \frac{n^{\frac{m}{2}-1}}{(\frac{m}{2}-1)!}\\
\end{split}
\end{equation*}
Simplifying
\begin{equation*}
\begin{split}
f_m(n)&\leq (\frac{14}{15})^{\frac{m}{6}}\frac{n^{\frac{m}{2}}}{(\frac{m}{2})!}\bigg[1-\frac{m}{2^{\frac{m}{2}}} \bigg\{1 -(\frac{14}{15})^{-\frac{1}{3}}\bigg\}- \frac{m^2}{2^{\frac{m}{2}+2}} \bigg\{1- (\frac{14}{15})^{\frac{1}{3}}\bigg\}\bigg]+\bigg[\frac{2n^{\frac{m}{2}-1}\log (\frac{n}{2})}{2^{\frac{m}{2}-1}}\bigg]e\\
&\qquad + \frac{2 n^{\frac{m}{2}-2}\log ^2 (\frac{n}{2})}{2^{\frac{m}{2}-2}}\frac{m}{4} + \frac{n^{\frac{m}{2}-1}}{2^{\frac{m}{2}-1}}\\
&\leq (\frac{14}{15})^{\frac{m}{6}}\frac{n^{\frac{m}{2}}}{(\frac{m}{2})!}\bigg[1-\frac{m}{2^{\frac{m}{2}}} \bigg\{1 -(\frac{14}{15})^{-\frac{1}{3}}\bigg\}- \frac{m^2}{2^{\frac{m}{2}+2}} \bigg\{1- (\frac{14}{15})^{\frac{1}{3}}\bigg\}\bigg]+\bigg[\frac{2\cdot e \cdot n^{\frac{m}{2}-1}\log n}{2^{\frac{m}{2}-1}}\\
&\qquad - \frac{2\cdot e \cdot n^{\frac{m}{2}-1}}{2^{\frac{m}{2}-1}}\bigg]+ \frac{2 n^{\frac{m}{2}-2}\log ^2 (\frac{n}{2})}{2^{\frac{m}{2}-2}}\frac{m}{4} + \frac{n^{\frac{m}{2}-1}}{2^{\frac{m}{2}-1}}\\
&\leq (\frac{14}{15})^{\frac{m}{6}}\frac{n^{\frac{m}{2}}}{(\frac{m}{2})!}+\frac{2e\cdot n^{\frac{m}{2}-1}\log n}{2^{\frac{m}{2}-1}}+\frac{mn^{\frac{m}{2}-2}\log ^2 (\frac{n}{2})}{2 \cdot 2^{\frac{m}{2}-2}}\\
&\leq (\frac{14}{15})^{\frac{m}{6}}\frac{n^{\frac{m}{2}}}{(\frac{m}{2})!}+n^{\frac{m}{2}-1}\log n\bigg[\frac{2e}{2^{\frac{m}{2}-1}}+\frac{m \log (\frac{n}{2})}{n \cdot 2^{\frac{m}{2}-1}}\bigg]\\
\end{split}
\end{equation*}
For all values of $m \geq 8$,  $\bigg[\frac{2e}{2^{\frac{m}{2}-1}}+\frac{m \log (\frac{n}{2})}{n \cdot 2^{\frac{m}{2}-1}}\bigg] \leq 1$. Therefore,
\begin{equation*}
\begin{split}
f_m(n) &\leq (\frac{14}{15})^{\frac{m}{6}}\frac{n^{\frac{m}{2}}}{(\frac{m}{2})!}+n^{\frac{m}{2}-1}\log n\\
\end{split}
\end{equation*}

For the case that $m$ is a multiple of $2$ but not $4$, an argument similar to the above can be used to show that $f_m(n) \leq (\frac{14}{15})^{\frac{m}{6}}\frac{n^{\frac{m}{2}}}{(\frac{m}{2})!}+n^{\frac{m}{2}-1}\log n$.

\end{proof}

We now prove the stated bound for $c_{125}$.
\begin{lemma}
For any S and T, and even a and b, ${S \choose a} \times {T \choose b+1} \cup {S \choose a+1} \times {T \choose b}$ can be exactly covered using ${|S| \choose \frac{a}{2}} \cdot {|T| \choose \frac{b}{2}}$ blocks.
\end{lemma}
\begin{proof}
Order the elements of $S$ and $T$. Pick $\frac{a}{2}$ elements of $S$, say $s_{i_1},s_{i_2},\cdots,s_{i_{\frac{a}{2}}}$, and $\frac{b}{2}$ elements of $T$, say $t_{j_1},t_{j_2},\cdots,t_{j_{\frac{b}{2}}}$. We associate a block corresponding to these sets as follows:
\begin{equation*}
\begin{split}
&\{s_1,\cdots,s_{i_1-1}\},\{s_{i_1}\},\cdots ,\{s_{i_{\frac{a}{2}-1}+1},\cdots,s_{i_{\frac{a}{2}}-1}\},\{s_{i_{\frac{a}{2}}}\},\\
&\{t_1,\cdots,t_{j_1-1}\},\{t_{j_1}\},\cdots ,\{t_{j_{\frac{b}{2}-1}+1},\cdots,t_{j_{\frac{b}{2}}-1}\},\{t_{j_{\frac{b}{2}}}\},\{s_{i_{\frac{a}{2}}+1},\cdots,s_p,t_{j_{\frac{b}{2}}+1},\cdots,t_q\}.
\end{split}
\end{equation*}
Among these take only the blocks which have $a+b+1$ parts.
Note that these form a disjoint cover.
\end{proof}
\begin{lemma}
For large S and T, and even a and b, ${S \choose a} \times {T \choose b+1} \cup {S \choose a+1} \times {T \choose b}$ can be exactly covered using $[(\frac{14}{15})^{\frac{b}{6}}+(\frac{14}{15})^{\frac{a}{6}}](1+o(1)){|S| \choose \frac{a}{2}} \cdot {|T| \choose \frac{b}{2}}$ blocks. Here the $o(1)$ term is as $|S| \hspace{1mm} {\&} \hspace{1mm} |T|$ go to $\infty$.
\end{lemma}
\begin{proof}
The hypergraph ${S \choose a} \times {T \choose b+1}$ can be exactly covered using $f_a(|S|)\cdot f_{b+1}(|T|)$ blocks. By Theorem 1, this is at most $(\frac{14}{15})^{\frac{a}{6}}(1+o(1)){|S| \choose \frac{a}{2}}{|T| \choose \frac{b}{2}}$ blocks. Likewise, ${S \choose a+1} \times {T \choose b}$ can be exactly covered using $(\frac{14}{15})^{\frac{b}{6}}(1+o(1)){|S| \choose \frac{a}{2}} {|T| \choose \frac{b}{2}}$ blocks.
\end{proof}
\begin{lemma}
For any odd r=4d+1, if ${\frac{n}{2} \choose \floor{\frac{r}{2}}} \times {\frac{n}{2} \choose \ceil{\frac{r}{2}}} \cup {\frac{n}{2} \choose \ceil{\frac{r}{2}}} \times {\frac{n}{2} \choose \floor{\frac{r}{2}}}$ can be covered using $\alpha (1+o(1)) {\frac{n}{2} \choose \floor{\frac{r}{4}}}^2$ blocks s.t. $\alpha <1$, then $f_r(n) \leq c_r(\alpha)\cdot(1+o(1)){n \choose \floor{\frac{r}{2}}}$ where $c_r(\alpha) <1$.
\end{lemma}
\begin{proof}
Recall the inequality.
\begin{equation*}
\begin{split}
f_r(n) \leq 2\cdot f_r(\frac{n}{2})+2\cdot f_1(\frac{n}{2})\cdot f_{r-1}(\frac{n}{2})+ \cdots +2 \cdot f_{\frac{r-1}{2}}(\frac{n}{2}) \cdot f_{\frac{r+1}{2}}(\frac{n}{2})\\
\end{split}
\end{equation*}
Pairing up two consecutive terms each and using Lemma 1 for each pair we have:
\begin{equation*}
\begin{split}
f_r(n) &\leq 2 (1+o(1)) \bigg[ \sum_{i=0}^{\frac{r-5}{4}}{\frac{n}{2} \choose i}{\frac{n}{2} \choose {\frac{r-1}{2}}-i}\bigg]+ 2 f_{\floor{\frac{r}{2}}}(\frac{n}{2})f_{\ceil{\frac{r}{2}}}(\frac{n}{2})\\
\end{split}
\end{equation*}

Using the hypothesis and by adding and subtracting ${\frac{n}{2} \choose \floor{\frac{r}{4}}}^2$ to the above equation we have:

\begin{equation*}
\begin{split}
f_r(n) &\leq 2 (1+o(1)) \bigg[ \sum_{i=0}^{\frac{r-5}{4}}{\frac{n}{2} \choose i}{\frac{n}{2} \choose {\frac{r-1}{2}}-i}\bigg]+ {\frac{n}{2} \choose \floor{\frac{r}{4}}}^2+2 f_{\floor{\frac{r}{2}}}(\frac{n}{2})f_{\ceil{\frac{r}{2}}}(\frac{n}{2})-{\frac{n}{2} \choose \floor{\frac{r}{4}}}^2\\
&\leq (1+o(1)){n \choose \floor{\frac{r}{2}}}+ 2 f_{\floor{\frac{r}{2}}}(\frac{n}{2})f_{\ceil{\frac{r}{2}}}(\frac{n}{2})-{\frac{n}{2} \choose \floor{\frac{r}{4}}}^2\\
&\leq (1+o(1)){n \choose \floor{\frac{r}{2}}}-(1-\alpha)(1+o(1)){\frac{n}{2} \choose \floor{\frac{r}{4}}}^2\\
&\leq [1-\frac{(1-\alpha)}{e^{\frac{r}{2}}}]{n \choose \floor{\frac{r}{2}}}\\
\end{split}
\end{equation*}
\end{proof}
\begin{lemma}
$f_{125}(n) \leq c_{125}(1+o(1)){n \choose 62}$, for a constant $c_{125} < 1$.
\end{lemma}

\begin{proof}
The hypergraph ${{S} \choose {63}} \times {{T} \choose {62}} \cup {{S} \choose {62}} \times {{T} \choose {63}}$ can be exactly covered using at most $2\cdot f_{62}(\frac{n}{2})f_{63}(\frac{n}{2})$ blocks. Using Lemma 2 we have 
\begin{equation*}
\begin{split}
2 \cdot f_{62}(\frac{n}{2}) f_{63}(\frac{n}{2})&\leq 2\cdot(\frac{14}{15})^{\frac{62}{6}}(1+o(1)){\frac{n}{2} \choose 31}^2\\
& \leq 0.981 \cdot (1+o(1)){\frac{n}{2} \choose 31}^2
\end{split}
\end{equation*}
The result follows from Lemma 3.
\end{proof}
As a consequence of Lemma 4, we have $c_{125}<1$. In fact solving the recurrence exactly for Theorem 1 using a computer program yields $c_{113} < 1$. 
\nocite{leader2017improved}
\bibliographystyle{plain}
\bibliography{reference}
\end{document}